\documentclass[11pt,reqno]{amsart}

\usepackage{amsfonts, amsthm, amsmath}
\allowdisplaybreaks[4]

\usepackage{rotating}

\usepackage{tikz}

\usepackage{graphics}

\usepackage{amssymb}

\usepackage{amscd}

\usepackage[latin2]{inputenc}

\usepackage{t1enc}

\usepackage[mathscr]{eucal}

\usepackage{indentfirst}

\usepackage{graphicx}

\usepackage{graphics}

\usepackage{pict2e}

\usepackage{mathrsfs}

\usepackage{enumerate}
\usepackage{ferrers} 
\usepackage[pagebackref]{hyperref}
\hypersetup{colorlinks=true}
\usepackage{cite}
\usepackage{color}
\usepackage{epic}
\usepackage{hyperref} 
\usepackage{framed}
\usepackage{mathabx}

\numberwithin{equation}{section}
\theoremstyle{plain}

\newtheorem{theorem}{Theorem}[section]

\newtheorem{corollary}[theorem]{Corollary}

\newtheorem{proposition}[theorem]{Proposition}

\theoremstyle{definition}

\newtheorem{Def}[theorem]{Definition}

\newtheorem{example}[theorem]{Example}

\newtheorem{remark}[theorem]{Remark}

\newtheorem{?}[theorem]{Problem}

\newcommand{\A}{\mathcal{A}}
\newcommand{\B}{\mathcal{B}}
\newcommand{\D}{\mathcal{D}}
\renewcommand{\L}{\mathcal{L}}
\renewcommand{\P}{\mathcal{P}}
\newcommand{\W}{\mathcal{W}}

\begin{document}

\title{Combinatorial proofs and refinements of three partition theorems of Andrews}

\author[S. Fu]{Shishuo Fu}
\address[Shishuo Fu]{College of Mathematics and Statistics, Chongqing University, Chongqing 401331, P.R. China}
\email{fsshuo@cqu.edu.cn}

\date{\today}

\begin{abstract}
In his recent work, Andrews revisited two-color partitions with certain restrictions on the differences between consecutive parts, and he established three theorems linking these two-color partitions with more familiar kinds of partitions. In this note, we provide bijective proofs as well as refinements of those three theorems of Andrews. Our refinements take into account the numbers of parts in each of the two colors.
\end{abstract}


\maketitle
\section{Introduction}

In his recent paper \cite{and21} published in the Hardy-Ramanujan Journal, Andrews revisited two-color partitions and obtained three new partition theorems that are reminiscent of the two theorems on two-color partitions proved in his 1987 paper \cite{and87}. In order to state these new theorems, we first recall some definitions in the theory of partitions.

Given a non-negative integer $n$, a {\it partition}, say $\lambda$, of $n$ is a non-increasing list of positive integers that sum up to $n$. We shall write $\lambda=\lambda_1+\lambda_2+\cdots$ with $\lambda_1\ge \lambda_2\ge\cdots$, where $|\lambda|=n$ might also be referred to as the {\it weight} of $\lambda$. In particular, we view the empty partition, denoted as $\epsilon$, as the one and only partition of $n=0$. Denote by $\P$ and $\P(n)$ the set of all partitions and the set of all partitions of $n$, respectively. For example, we have
$$\P(4)=\{4,~3+1,~2+2,~2+1+1,~1+1+1+1\},$$
and $2+1+1$ is said to be a partition of $4$ with one part of $2$ and two parts of $1$. A standard way of representing partitions pictorially is to use the so-called {\it Ferrers graph}~\cite[p.~7]{andtp}. For a partition $\lambda=\lambda_1+\lambda_2+\cdots$, we draw a left-justified array of cells, so that a top-down counting reveals $\lambda_i$ cells in the $i$-th row; see Fig~\ref{fig:profile} for an example\footnote{This way of drawing Ferrers graphs is usually called the ``English notation''.}. Using Ferrers graph, the {\it conjugate partition} of $\lambda$, usually denoted as $\lambda'$, is defined to be the partition rendered from counting cells in the Ferrers graph of $\lambda$ by columns in stead of rows. So for instance, the conjugate of $2+1+1$ is $3+1$.

{\it Two-color partitions}, on the other hand, are those in which each part may come in two colors, say red and green. We indicate the color of any given part by assigning a subindex $r$ (for red) or $g$ (for green) to that part. For instance, there are ten two-color partitions of $3$, as listed below.
\begin{align*}
& 3_r, 3_g, 2_r+1_r, 2_r+1_g, 2_g+1_r, 2_g+1_g,\\
& 1_r+1_r+1_r, 1_r+1_r+1_g, 1_r+1_g+1_g, 1_g+1_g+1_g.
\end{align*}
Given a two-color partition, we shall say that two parts are distinct if they are of different colors or different numerical values or both. We shall say that two parts are numerically distinct if they have different numercial values. As noted by Lovejoy \cite[p.~395]{lov03}, if we require in addition that all green parts are distinct, we essentially recover {\it overpartitions} \cite{CJ}, a well-studied variant that usually enjoys parallel results to those known ones for ordinary partitions. 

Suppose $n\ge 0$ and $d\ge 1$ are integers. Following Andrews \cite{and21}, we define $\L_d(n)$ to be the set of two-color partitions of $n$ into numerically distinct parts such that the following three conditions are satisfied:
\begin{enumerate}
  \item each red part is at least $d$ larger than the next largest part;
  \item each green part is at least $d+1$ larger than the next largest part;
  \item neither $1_g$ nor $(d-1)_g$ is allowed as a part.
\end{enumerate}
We use $L_d(n)$ to denote the cardinality of $\L_d(n)$.

The three theorems on two-color partitions proved in \cite{and21} are these:

\begin{theorem}\label{thm:L1}
$L_1(n)$ equals the number of two-color partitions of $n$ in which parts with the same color are distinct and green parts are all even numbers.
\end{theorem}

\begin{theorem}\label{thm:L2}
$L_2(n)$ equals the number of basis partitions of $n$.
\end{theorem}

\begin{theorem}\label{thm:L3}
$L_3(n)$ equals the number of partitions of $n$ into distinct parts.
\end{theorem}

\begin{remark}
To be precise, the original statements of Theorems~\ref{thm:L1} and \ref{thm:L3} in Andrews's paper \cite{and21} are different from what we give here, but it suffices to use Euler's celebrated ``Odd-distinct Theorem'' to bridge the gap. The definition of basis partition is a bit involved so we decide to postpone it to section~\ref{sec:L2} (see Definition~\ref{def:basis}).
\end{remark}

Andrews \cite{and21} proved the above theorems in a uniform way by manipulating generating functions and he requested for bijective proofs. The main purpose of this paper is to prove all three theorems bijectively. Indeed, our bijective approach makes it effortless to keep track of various partition statistics and consequently leads us to the following three refinements. We begin with two definitions.

\begin{Def}
Let $\L_d(n,k,\ell)$ (resp.~$L_d(n,k,\ell)$) denote the set (resp.~the number) of two-color partitions in $\L_d(n)$ with $k$ red parts and $\ell$ green parts. 
\end{Def}
Clearly, $L_d(n)=\sum_{k,\ell\ge 0}L_d(n,k,\ell)$, hence the following three results are indeed refinements of Theorems \ref{thm:L1}, \ref{thm:L2}, and \ref{thm:L3}, respectively.

\begin{Def}\label{def:A}
Let $\A(n,k,\ell)$ (resp.~$A(n,k,\ell)$) be the set (resp.~the number) of two-color partitions of $n$, each of which is consisted of $k+j$ distinct red parts and $\ell$ distinct even green parts for a certain non-negative integer $j$, wherein exactly $k$ red parts are larger than $\ell$.
\end{Def}

\begin{theorem}\label{thm:L1-ref}
For non-negative integers $n,k,\ell$, we have
$L_1(n,k,\ell)=A(n,k,\ell)$.
\end{theorem}
\begin{example}
Taking $(n,k,\ell)=(7,1,1)$, we see that $$\L_1(7,1,1)=\{6_g+1_r,~5_g+2_r,~5_r+2_g,~4_r+3_g\}.$$ On the other hand, there are also four two-color partitions in $\A(7,1,1)$, namely,
$$\A(7,1,1)=\{4_g+2_r+1_r,~4_r+2_g+1_r,~4_g+3_r,~5_r+2_g\}.$$
\end{example}

\begin{theorem}\label{thm:L2-ref}
The number of basis partitions $\lambda=(k+\ell,\pi,\sigma)$ of $n$ such that $\pi$ has exactly $\ell$ distinct parts, is given by $L_2(n,k,\ell)$.
\end{theorem}
\begin{example}
For $(n,k,\ell)=(15,1,2)$, there are six basis partitions meeting the requirements, namely (in terms of the triple expression $(d,\pi,\sigma)$; see Sect.~\ref{sec:L2}),
$$(3,3+1+1+1,\epsilon), (3,2+2+1+1,\epsilon), (3,2+1+1+1+1,\epsilon), (3,3+2,1), (3,3+1,2), (3,2+1,3).$$
On the other hand, one checks that
$$\L_2(15,1,2)=\{10_g+4_g+1_r,9_g+5_g+1_r,8_g+5_g+2_r,9_g+4_r+2_g,8_g+5_r+2_g,8_r+5_g+2_g\},$$
which also contains six partitions.
\end{example}

\begin{theorem}\label{thm:L3-ref}
$L_3(n,k,\ell)$ equals the number of partitions of $n$ into $k+2\ell$ distinct parts such that the Durfee square is of side $k+\ell$.
\end{theorem}

The definition of the Durfee square of a partition will be given in section~\ref{sec:L2}.
\begin{example}
For $(n,k,\ell)=(18,2,1)$, the following are the ten strict partitions meeting the requirements,
\begin{align*}
& 10+4+3+1,~9+5+3+1,~8+6+3+1,~8+5+4+1,~7+6+4+1, \\
& 9+4+3+2,~8+5+3+2,~7+6+3+2,~7+5+4+2,~6+5+4+3.
\end{align*}
On the other hand, one checks that
\begin{align*}
\L_3(18,2,1) &= \{13_g+4_r+1_r,~12_g+5_r+1_r,~12_r+5_g+1_r,~11_g+6_r+1_r,~11_r+6_g+1_r,\\
&\qquad 10_r+7_g+1_r,~11_g+5_r+2_r,~10_g+6_r+2_r,~10_r+6_g+2_r,~9_r+6_r+3_g\},
\end{align*}
which contains ten partitions as well.
\end{example}

We will prove one refinement in each of the ensuing sections, and we conclude the paper with some outlook for future research.

\section{Two proofs of Theorem~\ref{thm:L1-ref}}\label{sec:L1}

We present in this section two proofs of Theorem~\ref{thm:L1-ref}. The first proof is via generating functions and requires the following {\it $q$-binomial theorem} \cite[p.~17,~Thm.~2.1]{andtp}. For $|q|<1$, $|t|<1$, we have
\begin{align}\label{id:q-bin}
\sum_{n=0}^{\infty}\frac{(a)_nt^n}{(q)_n}=\frac{(at)_{\infty}}{(t)_{\infty}}.
\end{align}
Here and in the sequel, we use the standard $q$-series notations
\begin{align*}
&(a)_n =(a;q)_n =(1-a)(1-aq)\cdots(1-aq^{n-1}),\\
&(a)_{\infty} =(a;q)_{\infty} =\lim_{n\to \infty}(a;q)_{n},\text{ and } (a)_0 =1.
\end{align*}

We claim the following two generating functions for $L_1(n,k,\ell)$ and $A(n,k,\ell)$, respectively.
\begin{align}
\label{id:gfL1}
\sum_{n,k,\ell\ge 0}L_1(n,k,\ell)x^ky^{\ell}q^n &= \sum_{m\ge 0}\frac{(-yq/x)_m x^m q^{\binom{m+1}{2}}}{(q)_m},\\
\label{id:gfA}
\sum_{n,k,\ell\ge 0}A(n,k,\ell)x^ky^{\ell}q^n &= \sum_{k,\ell\ge 0}\frac{x^k y^{\ell}q^{\binom{k+\ell+1}{2}+\binom{\ell+1}{2}}}{(q)_k(q)_{\ell}}.
\end{align}

Indeed, given a partition $\lambda\in\L_1(n,k,\ell)$, suppose it has $m$ ($=k+\ell$) parts in total (regardless of the colors). These $m$ parts must all be distinct so at least they should contribute $1+2+\cdots+m=m(m+1)/2$ to the weight of $\lambda$. But the difference $\lambda_i-\lambda_{i+1}$ might as well be larger than $1$, and this ``widening of gaps'', if any, is accounted for by the factor
$$\frac{1}{(q)_m}=\frac{1}{1-q}\frac{1}{1-q^2}\cdots\frac{1}{1-q^m},$$
where $1/(1-q)$ dictates the widening between $\lambda_1$ and $\lambda_2$, $1/(1-q^2)$ determines the extra gap between $\lambda_2$ and $\lambda_3$, and so on and so forth. Finally, to assign the color for part, say $\lambda_i$, we make a decision between red (contributing $x$), and green (contributing $yq^i$, since each green part is at least $2$ larger than the next part). Collectively, this yields the factor
$$(x+yq)(x+yq^2)\cdots(x+yq^m),$$
and we arrive at the right hand side of \eqref{id:gfL1} after pulling out $x$ from each factor $(x+yq^i)$, $1\le i\le m$.

Next, to verify \eqref{id:gfA}, we see that for a given partition $\mu\in\A(n,k,\ell)$, its $k$ red parts larger than $\ell$ are generated by
$$\frac{x^kq^{(\ell+1)+(\ell+2)+\cdots+(\ell+k)}}{(q)_k}=\frac{x^kq^{k\ell+\binom{k+1}{2}}}{(q)_k}.$$
The remaining red parts, if any, must be no greater than $\ell$ and are generated by $(-q)_{\ell}$, while the $\ell$ green even parts of $\mu$ are generated by
$$\frac{y^{\ell}q^{2+4+\cdots+2\ell}}{(q^2;q^2)_{\ell}}=\frac{y^{\ell}q^{\ell(\ell+1)}}{(q)_{\ell}(-q)_{\ell}}.$$
Cancelling out $(-q)_{\ell}$, we arrive at the right hand side of \eqref{id:gfA}.

\begin{proof}[Analytic proof of Theorem~\ref{thm:L1-ref}]
It suffices to show that the right hand sides of \eqref{id:gfL1} and \eqref{id:gfA} are the same. Starting from the right hand side of \eqref{id:gfA} and substituting $m$ for $k+\ell$, we deduce that
\begin{align*}
\sum_{k,\ell\ge 0}\frac{x^k y^{\ell}q^{\binom{k+\ell+1}{2}+\binom{\ell+1}{2}}}{(q)_k(q)_{\ell}} &= \sum_{m\ge 0}\frac{x^mq^{\binom{m+1}{2}}}{(q)_m}\sum_{\ell=0}^{m}\frac{(q^{m-\ell+1})_{\ell}}{(q)_{\ell}}\left(\frac{y}{x}\right)^{\ell}q^{\binom{\ell+1}{2}}\\
&= \sum_{m\ge 0}\frac{x^mq^{\binom{m+1}{2}}}{(q)_m}\sum_{\ell=0}^m \frac{(q^{-m})_{\ell}}{(q)_{\ell}}\left(-\frac{yq^{m+1}}{x}\right)^{\ell}\\
&= \sum_{m\ge 0}\frac{x^mq^{\binom{m+1}{2}}}{(q)_m}\frac{(-yq/x)_{\infty}}{(-yq^{m+1}/x)_{\infty}}\quad \text{(by \eqref{id:q-bin} with $a=q^{-m}$, $t=-\frac{yq^{m+1}}{x}$)}\\
&= \sum_{m\ge 0}\frac{x^mq^{\binom{m+1}{2}}}{(q)_m}(-yq/x)_{m},
\end{align*}
which is the right hand side of \eqref{id:gfL1}, as desired.
\end{proof}

It is worth noting that there is an auxiliary parameter $j$\footnote{This parameter $j$ appeared implicitly in the bijection of Alladi and Gordon; see~\cite[p.~26, Sect.~4.4]{pak06}.} in the definition of $\A(n,k,\ell)$ (see Definition~\ref{def:A}), so one may ask what is the corresponding parameter for those two-color partitions in $\L_1(n,k,\ell)$. Our second proof of Theorem~\ref{thm:L1-ref}, which is to construct a bijection between $\L_1(n,k,\ell)$ and $\A(n,k,\ell)$, answers this question completely.

Andrews's original proof of Theorem~\ref{thm:L1} relies on the following identity of Lebesgue \cite{leb40}:
\begin{align}
\sum_{n\ge 0}\frac{(-zq;q)_n}{(q;q)_n}q^{\binom{n+1}{2}}=(-q;q)_{\infty}(-zq^2;q^2)_{\infty}.
\end{align}
In view of this, the first step towards proving the refinement Theorem~\ref{thm:L1-ref} bijectively, is to interprete the Lebesgue identity combinatorially. This has been done several times via different approaches; see for example the works of Alladi-Gordon~\cite{AG}, and Bessenrodt~\cite{bes94}, see also Pak's survey~\cite{pak06} where both proofs were nicely summarized. Later proofs, refinements, and polynomial analogues of Lebesgue identity include Alladi-Berkovich~\cite{AB}, Little-Sellers~\cite{LS09}, Rowell~\cite{row10}, Chen-Hou-Sun~\cite{CHS}, and Dousse-Kim~\cite{DK}. Our bijection, when interpreted in terms of tilings, is equivalent to Little and Sellers's construction \cite{LS09}, which itself has its roots in Alladi and Gordon's work~\cite{AG}. 

Little and Sellers's approach via tilings~\cite{LS09} is intuitive and lends itself to further applications~\cite{LS10}. However, for our convenience, we use words made of three letters $x,y$, and $z$, which correspond to black squares, white squares, and dominoes, respectively. 

The {\it profile} of an integer partition is the south-west to north-east border path in its Ferrers graph. This notion seems to be first introduced by Keith and Nath~\cite{KN}; see also \cite{FT18} and \cite{LXY}, where this way of representing integer partitions proved to be useful. For our purpose here, we write the profiles of partitions in $\L_1(n)$ not as words composed of two letters ($N$ for north and $E$ for east) as it was done in \cite{FT18}, but rather as words composed of three letters $x,y$, and $z$, so that we can tell red parts from green parts. More precisely, as we trace out the profile, we label the steps according to the following rules.
\begin{enumerate}
  \item the two consecutive steps $EN$ forming the corner cell of a certain red part are labeled together as $x$;
  \item the three consecutive steps $EEN$ forming the last two cells of a certain green part are labeled together as $z$;
  \item all remaining east steps are labeled as $y$.
\end{enumerate}
It should be clear from the definition that the profile in terms of the $x$-$y$-$z$ word as described above is uniquely determined by the partition and vice versa. The readers are invited to use the example in Fig.~\ref{fig:profile} to make sure they understand this correspondence. Let $A=\{x,y,z\}$ be our alphabet, and $A^*$ stands for the free monoid generated by the letters from $A$. We denote $\W$ the set of words in $A^*$ that are either empty or end with $x$ or $z$. For each word $u=u_1u_2\cdots u_m\in\W$, we define its {\it weight} to be
\begin{align}\label{word weight}
\omega(u):=\sum_{i=1}^m \chi(u_i\neq y)\cdot (i+|\{j\le i:u_j=z\}|),
\end{align}
where $\chi(S)=1$ if the statement $S$ is true and $\chi(S)=0$ otherwise. For the word $u=xzxyxzyyz$ in Fig.~\ref{fig:profile}, one checks that $\omega(u)=34$. We denote $\W(n,k,\ell)$ the set of words in $A^*$ having weight $n$, wherein the letter $x$ appears $k$ times and the letter $z$ appears $\ell$ times.

\begin{proposition}\label{prop:bij-phi}
 The mapping, say $\phi$, that sends each partition to its profile written as a word in $x,y$, and $z$, is a bijection from $\L_1(n,k,\ell)$ to $\W(n,k,\ell)$.
\end{proposition}

\begin{figure}[ht]
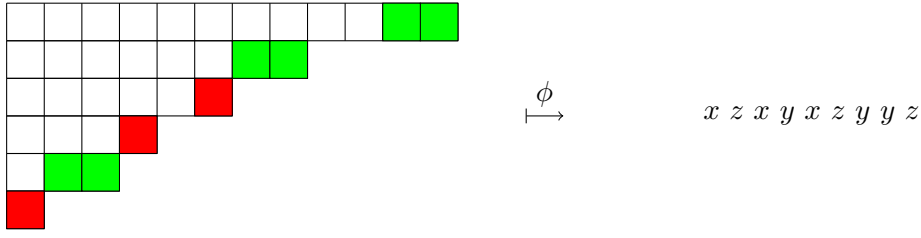

\begin{ferrers}
\addcellrows{12+8+6+4+3+1}
\highlightcellbycolor{1}{12}{green}
\highlightcellbycolor{1}{11}{green}
\highlightcellbycolor{2}{7}{green}
\highlightcellbycolor{2}{8}{green}
\highlightcellbycolor{5}{2}{green}
\highlightcellbycolor{5}{3}{green}
\highlightcellbycolor{3}{6}{red}
\highlightcellbycolor{4}{4}{red}
\highlightcellbycolor{6}{1}{red}
\transformto{$\phi$}
\addtext{10.7}{-1.5}{$x~z~x~y~x~z~y~y~z$}
\end{ferrers}
\caption{the profile of partition $\lambda=12_g+8_g+6_r+4_r+3_g+1_r$.}
\label{fig:profile}
\end{figure}

Now to prove Theorem~\ref{thm:L1-ref}, it remains to construct a bijection, say $\psi$, from the set of words $\W(n,k,\ell)$, to the set of two-color partitions $\A(n,k,\ell)$. As alluded to earlier, we can decompose $\A(n,k,\ell)$ further into subsets $\A(n,k,\ell,j)$, $0\le j\le \ell$, i.e., the set of partitions in $\A(n,k,\ell)$ with precisely $j$ red parts no greater than $\ell$. Again, we use $A(n,k,\ell,j)$ to denote its cardinality. On the other hand for the profile words, we further distinguish two types of letter $z$ to account for the extra parameter $j$.
\begin{Def}
For a word $u=u_1u_2\cdots u_m\in A^*$, we say a letter $u_i=z$ is {\it odd (resp.~even)} if for all $1\le j< i$, there is an odd (resp.~even) number of $u_j$ being a $y$ or an odd $z$. For each $0\le j\le \ell$, we denote $\W(n,k,\ell,j)$ the set of words in $\W(n,k,\ell)$ that contain $j$ odd $z$'s.
\end{Def}
Take the word $u=xzxyxzyyz$ in Fig.~\ref{fig:profile} for example, from left to right, the first and the last $z$ are even while the middle $z$ is odd, hence $u\in\W(34,3,3,1)$. Now we are ready to construct the aforementioned bijection $\psi$. Then Theorem~\ref{thm:L1-ref} follows immediately as a direct corollary of Proposition~\ref{prop:bij-phi} and Theorem~\ref{thm:bij-psi}.

\begin{theorem}\label{thm:bij-psi}
There exists a bijection $\psi:\W(n,k,\ell,j)\to\A(n,k,\ell,j)$.
\end{theorem}

\begin{corollary}
The composition $\psi\circ\phi$ is a bijection from $\L_1(n,k,\ell)$ to $\A(n,k,\ell)$. In particular, Theorem~\ref{thm:L1-ref} holds true.
\end{corollary}

\begin{proof}[Proof of Theorem~\ref{thm:bij-psi}]
Take any word $u=u_1u_2\cdots u_m\in\W(n,k,\ell,j)$, we aim to derive a partition pair $\psi(u):=(\pi,\sigma)$, where $\pi$ is a distinct partition into $k+j$ parts while $\sigma$ is a distinct partition into $\ell$ even parts. 

First off, we screen the word $u$ from right to left looking for letter $z$, and suppose they are $u_{s_1},u_{s_2},\ldots,u_{s_{\ell}}$ in $u$ with $s_1>s_2>\cdots>s_{\ell}$. Next, for each $i=1,2,\ldots,\ell$, we set
\begin{align*}
\sigma_i &:= s_i-|\{t\le s_i:u_t=x\}|+|\{t\le s_i:\text{$u_t$ is an even $z$}\}|,\text{ and}\\
v_i &:= \begin{cases}
x, & \text{if $u_{s_i}$ is an odd $z$},\\
y, & \text{if $u_{s_i}$ is an even $z$}.
\end{cases}
\end{align*}
We further let $u'$ be the word obtained from $u$ by replacing all letter $z$ by letter $y$, and let $$\hat{u}:=v_1v_2\cdots v_{\ell}u'.$$
Finally, we take $\pi:=\phi^{-1}(\hat{u})$\footnote{Strictly speaking, the images under the mapping $\phi$ should be words ending with $x$ or $z$. So, when $\hat{u}$ ends with $y$, we simply ignore all its ending $y$'s (doing this will not change the weight of $\hat{u}$), and then apply $\phi^{-1}$.} and $\sigma:=\sigma_1+\sigma_2+\cdots+\sigma_{\ell}$, i.e., $\sigma$ is the partition with parts $\sigma_1,\sigma_2,\ldots,\sigma_{\ell}$. With the following claims, we see that $\psi(u)=(\pi,\sigma)$ is indeed a two-color partition in $\A(n,k,\ell,j)$.
\begin{enumerate}
  \item all $\sigma_i$ are even integers and $\sigma_i>\sigma_{i+1}$.
  \item $\hat{u}$ is an $x$-$y$ word containing $k+j$ copies of letter $x$, $j$ of which occur in the prefix $v_1v_2\cdots v_{\ell}$.
  \item $\omega(u)=\omega(\hat{u})+|\sigma|$.
\end{enumerate}
The proofs of the claims above are straightforward verifications. We give some details on the proof of claim (3) here and leave the rest to the interested readers. For the trivial case $\ell=0$, clearly $\hat{u}=u'=u$, and $\sigma=\epsilon$, so (3) holds true. In general for $\ell>0$, we make a letter-by-letter comparison between $u$ and $\hat{u}$. For a certain letter $u_t$ in $u$, if $u_t=y$, then according to \eqref{word weight} it contributes nothing in either $\omega(u)$ or $\omega(\hat{u})$. Otherwise $u_t\neq y$, we go through various ranges for the index $t$ and list in Table~\ref{interval} the corresponding increments in $u_t$'s contributions to the weights going from $u$ to $\hat{u}$.
\begin{table}[h]
{\small
\begin{tabular}{c|c|c|c|c|c|c|c}
\hline
$t>s_1$ & $t=s_1$ & $s_2<t<s_1$ & $t=s_2$ & $s_3<t<s_2$ & $\cdots$ & $t=s_{\ell}$ & $1\le t<s_{\ell}$\\
\hline
$0$ & $-(s_1+\ell)$ & $1$ & $-(s_2+\ell-1)$ & $2$ & $\cdots$ & $-(s_{\ell}+1)$ & $\ell$\\
\hline
\end{tabular}
}
\vspace{2mm}
\caption{Letterwise increments in weights from $u$ to $\hat{u}$.}
\label{interval}
\end{table}
In addition, for the prefix $v_1v_2\cdots v_{\ell}$ of $\hat{u}$, each $v_i$ would incur an increment of $i$ in weight, if and only if $u_{s_i}$ is an odd $z$ in $u$. Summing up all these changes, we have
\begin{align*}
\omega(\hat{u})-\omega(u) &= \sum_{i=1}^{\ell}-(s_i+\ell+1-i)+|\{t\le s_i:u_t=x\}|+\sum_{\substack{1\le i\le \ell \\ \text{$u_{s_i}$ is odd}}}i \\
&= \sum_{i=1}^{\ell}-s_i+|\{t\le s_i:u_t=x\}|-\sum_{\substack{1\le i\le \ell \\ \text{$u_{s_i}$ is even}}}i =-\sum_{i=1}^{\ell} \sigma_i,
\end{align*}
which yields claim (3).

Lastly, it should be clear from our construction (especially the definitions of $\sigma_i$ and $v_i$), that the mapping $\psi$ is invertible, thus a bijection between $\W(n,k,\ell,j)$ and $\A(n,k,\ell,j)$.
\end{proof}

For our running example $u=xzxyxzyyz$ from Fig.~\ref{fig:profile}, one checks that $\pi=8+6+4+2$ and $\sigma=8+4+2$, therefore $\psi(u)=(\pi,\sigma)\in\A(34,3,3,1)$.

\begin{?}
Recall our first proof of Theorem~\ref{thm:L1-ref}, where we have derived the generating function \eqref{id:gfA} for $A(n,k,\ell)$. The same analysis readily gives us the following generating function for the refined $A(n,k,\ell,j)$:
\begin{align*}
\sum_{n,k,\ell,j\ge 0}A(n,k,\ell,j)x^k y^{\ell} z^j q^n &= \sum_{k,\ell\ge 0}\frac{(-zq;q)_{\ell}}{(q;q)_k(q^2;q^2)_{\ell}}x^k y^{\ell}q^{\binom{k+\ell+1}{2}+\binom{\ell+1}{2}}.
\end{align*}
Therefore, it might be interesting to find the generating function for $|\W(n,k,\ell,j)|$, so as to give an analytic counterpart of Theorem~\ref{thm:bij-psi}.
\end{?}

\section{Proof of Theorem~\ref{thm:L2-ref}}\label{sec:L2}
We begin this section by introducing the {\it Durfee square}, which is an important conjugation invariant for integer partition. Suppose $\lambda=\lambda_1+\lambda_2+\cdots$ is a partition, then the Durfee square of $\lambda$ is the largest square that could fit in its Ferrers graph. Alternatively, $\lambda$ has a {\it Durfee square of side $d$}, if and only if $d$ is the largest integer $i$, such that $\lambda_i\ge i$. For example, the partition in Fig.~\ref{fig:profile} (ignoring its colors) has a Durfee square of side $4$. 

Now each partition $\lambda$ can be written as a triple $(d,\pi,\sigma)$, where $d$ is the Durfee side of $\lambda$, and $\pi=\lambda_{d+1}+\lambda_{d+2}+\cdots$ (resp.~$\sigma=\lambda'_{d+1}+\lambda'_{d+2}+\cdots$) is the subpartition below the Durfee square of $\lambda$ (resp.~the conjugate partition $\lambda'$). The following two facts are immediate from the definition.
\begin{itemize}
  \item $|\lambda|=d^2+|\pi|+|\sigma|$;
  \item both $\pi$ and $\sigma$ are partitions with parts no greater than $d$.
\end{itemize}
As our running example, the (uncolored) partition $\lambda=12+8+6+4+3+1$ can be represented as $(4,~3+1,~3+3+2+2+1+1+1+1)$.

Basis partitions were first introduced by Gupta~\cite{gup78} and his definition involves the notion of {\it successive ranks}~\cite[Sect.~9.3]{andtp}. All we need here is the following alternative definition due to Nolan, Savage, and Wilf~\cite{NSW}. See also \cite{hir99} and \cite{and15} for other works related to basis partitions.

\begin{Def}{\cite[Thm.~3]{NSW}}\label{def:basis}
A partition $\lambda=(d,\pi,\sigma)$ is said to be a {\it basis partition}, if and only if $\pi$ and $\sigma$ do not have parts in common.
\end{Def}

For a partition $\lambda=\lambda_1+\lambda_2+\cdots+\lambda_m \in\L_2(n)$, in addition to the standard Ferrers graph, we can also represent $\lambda$ via the {\it $2$-indented Ferrers graph}; see Fig.~\ref{fig:2indent}. We denote $\tilde{\lambda}$ the partition obtained from $\lambda$ by setting
$$\tilde{\lambda}_i:=\lambda_i-2m+2i-1, \text{ for $1\le i\le m$}.$$
Graphically, the $2$-indented Ferrers graph of $\lambda$ is the ordinary Ferrers graph of $\tilde{\lambda}$ with the staircase $(2m-1)+(2m-3)+\cdots+3+1$ attached to its left.

Next, we prove Theorem~\ref{thm:L2-ref} by constructing a direct bijection $\eta$ between $\L_2(n,k,\ell)$ and $\B(n,k,\ell)$, i.e., the set of basis partitions $\lambda=(k+\ell,\pi,\sigma)$ of $n$, wherein $\pi$ has $\ell$ distinct parts.

\begin{proof}[Proof of Theorem~\ref{thm:L2-ref}]
Given a partition $\lambda=\lambda_1+\lambda_2+\cdots+\lambda_m$ in $\L_2(n,k,\ell)$ (here $m=k+\ell$), we proceed to construct its image $\eta(\lambda)$. First we isolate $\tilde{\lambda}$ from its $2$-indented Ferrers graph, then we decompose its conjugate $\tilde{\lambda}'$ into two partitions, $\pi$ and $\sigma$. $\pi$ is consisted of all those parts in $\tilde{\lambda}'$ having the same length as a certain part containing a green cell (i.e., the last cell of a green part in $\lambda$), while the remaining parts of $\tilde{\lambda}'$ constitute $\sigma$. A moment of reflection should reveal that the above way of decomposition guarantees that $\pi$ and $\sigma$ do not have parts in common, and that $\pi$ has $\ell$ distinct parts. Therefore, $\eta(\lambda):=(m,\pi,\sigma)$ is indeed a basis partition in $\B(n,k,\ell)$. The whole process is clearly invertible so we see $\eta$ is a bijection.
\end{proof}

\begin{figure}[ht]
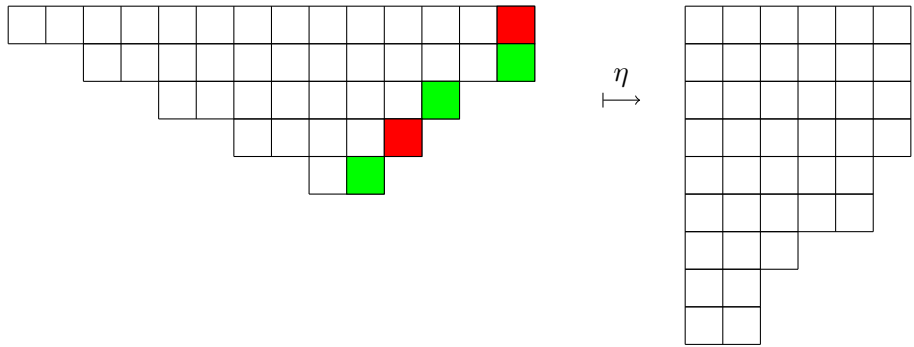

\begin{ferrers}
\addcellrow[0]{14}
\addcellrow[2]{12}
\addcellrow[4]{8}
\addcellrow[6]{5}
\addcellrow[8]{2}
\highlightcellbycolor{1}{14}{red}
\highlightcellbycolor{2}{14}{green}
\highlightcellbycolor{3}{12}{green}
\highlightcellbycolor{4}{11}{red}
\highlightcellbycolor{5}{10}{green}
\transformto{$\eta$}
\putright
\addcellrows{6+6+6+6+5+5+3+2+2}
\end{ferrers}
\caption{the $2$-indented Ferrers graph of $\lambda=14_r+12_g+8_g+5_r+2_g$ and its image $\eta(\lambda)$.}
\label{fig:2indent}
\end{figure}

We should point out that it is also not difficult to give a generating function proof of Theorem~\ref{thm:L2-ref}. The idea is to show that both types of partitions have the same generating function as follows:
\begin{align}\label{gf:L2ref}
\sum_{n,k,\ell\ge 0}\L_2(n,k,\ell)x^k y^{\ell} q^n &=\sum_{n,k,\ell\ge 0}\B(n,k,\ell)x^k y^{\ell} q^n= \sum_{m\ge 0}\frac{(-yq/x;q)_m x^{m}q^{m^2}}{(q;q)_m}.
\end{align}
Setting $x=y=1$ in \eqref{gf:L2ref} recovers the generating function of basis partitions, which was first derived by Nolan-Savage-Wilf~\cite[Coro.~2]{NSW}.

\section{Proof of Theorem~\ref{thm:L3-ref}}\label{sec:L3}

Let $\D(n,k,\ell)$ denote the set of those partitions as described in Theorem~\ref{thm:L3-ref}, i.e., strict partitions of $n$ into $k+2\ell$ parts such that the Durfee square is of side $k+\ell$. Our goal in this section is to construct a direct bijection $\theta:\L_3(n,k,\ell)\to\D(n,k,\ell)$. 

The first thing to notice is that $\L_3(n,k,\ell)\subset \L_2(n,k,\ell)$ from their definitions, but partitions in $\D(n,k,\ell)$ might not be basis partitions so in general $\D(n,k,\ell)\not\subset\B(n,k,\ell)$. Consequently, restricting the bijection $\eta$ onto $\L_3(n,k,\ell)$ is not sufficient. Nonetheless, the main ingredient in the construction of $\theta$ is analogous to that of $\eta$.

\begin{proof}[Proof of Theorem~\ref{thm:L3-ref}]
Given a non-empty partition $\lambda=\lambda_1+\cdots+\lambda_{m}\in\L_3(n,k,\ell)$ with $m=k+\ell$, we explain how to derive its image $\theta(\lambda)$. The first step is the same as that of $\eta$. Namely, we display $\lambda$ using its $2$-indented Ferrers graph and peel off the partition $\tilde{\lambda}$ with
$$\tilde{\lambda}_i:=\lambda_i-2m+2i-1,\text{ for }1\le i\le m.$$
Note that $\tilde{\lambda}$ is a strict ($\tilde{\lambda}_i>\tilde{\lambda}_{i+1}$) partition with either $m-1$ (when $\tilde{\lambda}_{m}=0$) or $m$ (when $\tilde{\lambda}_{m}\neq 0$) parts. Moreover, if $\lambda_i$ is a green part in $\lambda$, then $\tilde{\lambda}_i-\tilde{\lambda}_{i+1}\ge 2$.

Next, we decompose $\tilde{\lambda}'$ into two subpartitions $\pi$ and $\sigma$, where $\pi$ is a strict partition consisted of all those parts in $\tilde{\lambda}'$ containing a green cell, in other words, when $\lambda_i$ is a green part in $\lambda$, then $\lambda'_{\lambda_i}$ is a part assigned to $\pi$. The remaining parts of $\tilde{\lambda}$ form $\sigma$. It is not hard to verify the following facts.
\begin{itemize}
  \item $\pi$ is a strict partition with $\ell$ parts, all of which are no greater than $m$.
  \item If $\lambda_m=1$, then the largest part of $\pi$ is less than $m$.
  \item $\sigma'$ is a strict partition into $m-1$ or $m$ parts, depending on whether $\lambda_m=1$ or not.
  \item $|\lambda|=m^2+|\pi|+|\sigma|$.
\end{itemize}
These facts guarantee that $\theta(\lambda):=(m,\pi,\sigma)$ is indeed a partition in $\D(n,k,\ell)$, expressed using the triple notation (see section~\ref{sec:L2}). It should also be clear how to reverse the above process, so $\theta$ is bijective and the proof is now completed.
\end{proof} 

We note that in \cite{and21}, Andrews's first proof of Theorem~\ref{thm:L3} utilized Sylvester's 1882 identity~\cite[Thm.~9.2]{andtp}:
\begin{align}\label{id:Sylvester}
(-aq;q)_{\infty} &= 1+\sum_{k=1}^{\infty}\frac{a^kq^{(3k^2-k)/2}(-aq;q)_{k-1}(1+aq^{2k})}{(q;q)_k}.
\end{align}
Here the parameter $a$ keeps track of the number of parts in each strict partition. Seeing that $\D(n,k,\ell)$ actually refines the set of strict partitions further by the size of the Durfee square, one can view our Theorem~\ref{thm:L3-ref} as a combinatorial refinement of \eqref{id:Sylvester}. On the other hand, \eqref{id:Sylvester} has seen multi-dimensional extensions recently, first by Alladi~\cite{ali17}, then by Chern-Fu-Tang~\cite{CFT}. With this in mind, it seems plausible that we try to adapt our combinatorial approach to the multi-colored partition case, so as to give new partition-theoretical interpretations of those extended $q$-summations derived in \cite{ali17} and \cite{CFT}.

Another possible line of investigation, is to restrict $\L_d(n,k,\ell)$ further, by bounding the largest part of each two-color partition. This so-called ``finitization'' would presumably provide us with new polynomial analogues of the corresponding $q$-series identities, such as it has been done in \cite{BU} and \cite{unc21}.

\section*{Acknowledgement}
The author was partly supported by the National Natural Science Foundation of China grant 12171059 and the Natural Science Foundation Project of Chongqing (No.~cstc2021jcyj-msxmX0693).

\end{document}